\documentclass[reqno]{amsart}

\usepackage[T1]{fontenc}
\usepackage{amssymb}
\usepackage{enumitem}
\usepackage{mathrsfs}
\usepackage[colorlinks, linkcolor=blue, citecolor=blue, urlcolor=blue]{hyperref}
\usepackage{tikz}

\newtheorem{theorem}{Theorem}[section]
\newtheorem{corollary}[theorem]{Corollary}
\newtheorem{lemma}[theorem]{Lemma}

\theoremstyle{definition}

\DeclareMathOperator{\diag}{diag}

\title{A choice-free proof of Mal'cev's theorem on quasivarieties}

\author{Guozhen Shen}
\address{School of Philosophy\\
Wuhan University\\
No.~299 Bayi Road\\
Wuhan 430072\\
Hubei Province\\
People's Republic of China}
\email{shen\_guozhen@outlook.com}

\thanks{The author was partially supported by National Natural Science Foundation of China grant number 12101466.}



\begin{document}

\begin{abstract}
In 1966, Mal'cev proved that a class $\mathcal{K}$ of first-order structures with a specified signature is a quasivariety
if and only if $\mathcal{K}$ contains a unit and is closed under isomorphisms, substructures, and reduced products.
In this article, we present a proof of this theorem in $\mathsf{ZF}$ (the Zermelo--Fraenkel set theory without the axiom of choice).
\end{abstract}

\maketitle

\section{Introduction}
Let $\Sigma$ be a first-order signature, that is, a set of relation and function symbols, each having a fixed finite arity.
A \emph{basic Horn formula} is a formula in first-order logic of the form
\begin{equation}\label{sh01}
\phi_0\vee\cdots\vee\phi_n,
\end{equation}
where at most one of the formulas $\phi_i$ is an atomic formula, the rest being negations of atomic formulas.
We call the formula \eqref{sh01} \emph{strict} if at least one of the formulas $\phi_i$ is an atomic formula.
Atomic formulas and strict basic Horn formulas are also called \emph{identities} and \emph{quasi-identities},
respectively (see~\cite{Malcev1966}). A class $\mathcal{K}$ of $\Sigma$-structures is called a \emph{variety}
(or \emph{quasivariety}) if it is axiomatized by a set of identities (or quasi-identities), that is,
there exists a set $\Gamma$ of identities (or quasi-identities) such that $\mathcal{K}$ is the class
of all $\Sigma$-structures satisfying every formula in $\Gamma$. When $\Sigma$ contains no relation symbols,
$\Sigma$-structures are referred to as \emph{$\Sigma$-algebras}.

In 1935, Birkhoff~\cite{Birkhoff1935} proved that a class $\mathcal{K}$ of algebras of the same signature
is a variety if and only if $\mathcal{K}$ is closed under homomorphic images, subalgebras, and products.
This result is known as Birkhoff's HSP theorem and readily generalizes to first-order structures.
In 1966, Mal'cev~\cite{Malcev1966} proved that a class $\mathcal{K}$ of $\Sigma$-structures is a quasivariety
if and only if $\mathcal{K}$ contains a unit and is closed under isomorphisms, substructures, and reduced products.
Mal'cev's theorem on quasivarieties can be extended to the following result (see~\cite[Theorem~9.4.7]{Hodges1993}):
A class $\mathcal{K}$ of $\Sigma$-structures is axiomatized by a set of basic Horn formulas
if and only if $\mathcal{K}$ is closed under isomorphisms, substructures, and reduced products.

In~\cite{Andreka1981}, it is shown that the axiom of choice can be avoided in the proof of Birkhoff's HSP theorem.
The point is to use the collection principle (see~\cite[p.~65]{Jech2003}) instead of the axiom of choice.
In this article, using similar ideas, we show that Mal'cev's theorem on quasivarieties, as well as its extension stated above, 
can also be proved without the aid of the axiom of choice. We also present a new choice-free proof of Birkhoff's HSP theorem,
without using the concept of free algebras.

\section{Preliminaries}
Throughout this article, we work in $\mathsf{ZF}$ (i.e., the Zermelo--Fraenkel set theory without the axiom of choice).
To make this article self-contained, we provide in this section the necessary terminology and notation for further study.

\subsection{Notation from logic}
Fix a first-order signature $\Sigma$. For a formula $\phi$, we write $\phi(x_1,\dots,x_n)$ to indicate that
the free variables of $\phi$ are among $x_1,\dots,x_n$. When $n=0$, that is, when $\phi$ has no free variables,
$\phi$ is called a \emph{sentence}. If $A$ is a structure and $a_1,\dots,a_n\in A$,
we write $A\models\phi[a_1,\dots,a_n]$ to indicate that $a_1,\dots,a_n$ satisfy $\phi$ in $A$.
$A$ satisfies $\phi$, written $A\models\phi$, if and only if $A\models\phi[a_1,\dots,a_n]$ for all $a_1,\dots,a_n\in A$.
If $\mathcal{K}$ is a class of structures and $\Gamma$ is a set of formulas,
we write $\mathcal{K}\models\phi$ and $A\models\Gamma$ to indicate that every structure in $\mathcal{K}$ satisfies $\phi$,
and $A$ satisfies every formula in $\Gamma$, respectively.

Homomorphisms, embeddings, isomorphisms, substructures, and products are defined as usual.
A \emph{unit} is a structure $A$ consisting of a single element $a$, such that $r^A(a,\dots,a)$ holds
for all relation symbols $r$ in $\Sigma$. Note that the trivial product is a unit.
Reduced products are systematically studied in~\cite{Frayne1962}, where the following lemma is proved.

\begin{lemma}\label{sh02}
Let $\langle A_i\rangle_{i\in I}$ be a family of structures, let $\mathcal{F}$ be a filter over $I$,
and let $\prod_\mathcal{F}A_i$ denote the reduced product of $\langle A_i\rangle_{i\in I}$ modulo $\mathcal{F}$.
\begin{enumerate}[label=\upshape(\arabic*)]
  \item For every atomic formula $\phi(x_1,\dots,x_n)$ and every $\bar{a}_1,\dots,\bar{a}_n\in\prod A_i$,
        \[
        \prod_\mathcal{F}A_i\models\phi[\bar{a}_1/\mathcal{F},\dots,\bar{a}_n/\mathcal{F}]\iff
        \{i\in I\mid A_i\models\phi[\pi_i(\bar{a}_1),\dots,\pi_i(\bar{a}_n)]\}\in\mathcal{F},
        \]
        where $\pi_i$ denotes the $i$-th projection function.\label{sh08}
  \item For every basic Horn formula $\phi(x_1,\dots,x_n)$ and every $\bar{a}_1,\dots,\bar{a}_n\in\prod A_i$,\label{sh09}
        \[
        \{i\in I\mid A_i\models\phi[\pi_i(\bar{a}_1),\dots,\pi_i(\bar{a}_n)]\}\in\mathcal{F}\ \Longrightarrow\
        \prod_\mathcal{F}A_i\models\phi[\bar{a}_1/\mathcal{F},\dots,\bar{a}_n/\mathcal{F}].
        \]
\end{enumerate}
\end{lemma}
Note that the proof of Lemma~\ref{sh02} does not use the axiom of choice.

\subsection{Diagrams}
For a $\Sigma$-structure $A$, $\Sigma(A)$ denotes the signature obtained from $\Sigma$ by adding a new constant symbol $\dot{a}$ for each $a\in A$.
$\diag^+(A)$ ($\diag^-(A)$) denotes the set of all (negated) atomic $\Sigma(A)$-sentences satisfied by $\langle A,a\rangle_{a\in A}$.
$\diag(A)=\diag^+(A)\cup\diag^-(A)$ is called the \emph{diagram} of $A$.
\begin{lemma}\label{sh03}
Let $A$ be a $\Sigma$-structure and let $B$ be a $\Sigma(A)$-structure.
\begin{enumerate}[label=\upshape(\arabic*)]
  \item If $B\models\diag(A)$, then $A$ can be embedded into $B|_\Sigma$,
        where $B|_\Sigma$ denotes the $\Sigma$-reduct of $B$.\label{sh04}
  \item If $B\models\diag^-(A)$, then $A$ is a homomorphic image of a substructure of $B|_\Sigma$.\label{sh05}
\end{enumerate}
\end{lemma}
\begin{proof}
For \ref{sh04}, see~\cite[Lemma~1.4.2]{Hodges1993}.
For \ref{sh05}, if $B\models\diag^-(A)$, then it is easy to see that
\[
\{\langle t^B,t^{\langle A,a\rangle_{a\in A}}\rangle\mid t\text{ is a closed $\Sigma(A)$-term}\}
\]
is a homomorphism from the substructure of $B|_\Sigma$ generated by $\{\dot{a}^B\mid a\in A\}$ onto $A$.
\end{proof}

\subsection{The collection principle}
A \emph{class relation} is a class that consists of ordered pairs.
The collection principle, which is a theorem in $\mathsf{ZF}$, 
can be stated as the following lemma.

\begin{lemma}\label{sh06}
Let $R$ be a class relation. For every set $S$ such that for every $x\in S$ there is $y$ for which $\langle x,y\rangle\in S$,
there exists a set $I\subseteq R$ such that for every $x\in S$ there is $y$ for which $\langle x,y\rangle\in I$.
\end{lemma}
\begin{proof}
See~\cite[p.~65]{Jech2003}.
\end{proof}

\section{The main results}
In this section, we provide choice-free proofs of several well-known theorems.

\subsection{A choice-free proof of Mal'cev's theorem on quasivarieties}
We first prove the following result, which is \cite[Lemma~9.4.6]{Hodges1993}.
\begin{lemma}\label{sh07}
Let $\mathcal{K}$ a class of $\Sigma$-structures and let $\Delta$ denote
the set of all basic Horn formulas $\phi$ such that $\mathcal{K}\models\phi$.
For every $\Sigma$-structure $A$, if $A\models\Delta$, then $A$ can be embedded
into some reduced product of structures in $\mathcal{K}$.
\end{lemma}
\begin{proof}
Let $A$ be a $\Sigma$-structure such that $A\models\Delta$.
Let $\mathcal{K}'$ be the class of all $\Sigma(A)$-structures $B$ such that $B|_\Sigma\in\mathcal{K}$.
Let $\theta_0,\dots,\theta_{n-1}\in\diag^+(A)$ and let $\theta_n$ be an atomic sentence such that
$\neg\theta_n\in\diag^-(A)$. We claim that $\{\theta_0,\dots,\theta_{n-1},\neg\theta_n\}$ has a model in $\mathcal{K}'$.

We assume the contrary and aim for a contradiction. Then
\[
\mathcal{K}'\models\neg\theta_0\vee\dots\vee\neg\theta_{n-1}\vee\theta_n.
\]
Let $\phi(x_1,\dots,x_m)$ be a basic Horn formula and let $a_1,\dots,a_m$ be pairwise distinct elements of $A$
such that $\neg\theta_0\vee\dots\vee\neg\theta_{n-1}\vee\theta_n$ is $\phi(\dot{a}_1,\dots,\dot{a}_m)$.
Then it follows that $\mathcal{K}'\models\phi(\dot{a}_1,\dots,\dot{a}_m)$,
which implies that $\mathcal{K}\models\phi$ and thus $\phi\in\Delta$.
Hence, $A\models\phi$, which implies that $A\models\phi[a_1,\dots,a_m]$
and so $\langle A,a\rangle_{a\in A}\models\neg\theta_0\vee\dots\vee\neg\theta_{n-1}\vee\theta_n$,
contradicting that $\theta_0,\dots,\theta_{n-1}\in\diag^+(A)$ and $\neg\theta_n\in\diag^-(A)$.

Similarly, when $\diag^-(A)=\varnothing$, $\{\theta_0,\dots,\theta_{n-1}\}$ has a model in $\mathcal{K}'$. Let
\[
S=[\diag^+(A)]^{<\omega}\times[\diag^-(A)]^{\leqslant1}.
\]
By Lemma~\ref{sh06}, there exists a set
\[
I\subseteq\{\langle\Theta,\Xi,B\rangle\mid\langle\Theta,\Xi\rangle\in S,B\in\mathcal{K}'\text{ and }B\models\Theta\cup\Xi\}
\]
such that for every $\langle\Theta,\Xi\rangle\in S$ there is $B$ for which $\langle\Theta,\Xi,B\rangle\in I$.

For all $i\in I$, let $\Theta_i$, $\Xi_i$, and $B_i$ denote the first, second, and third coordinates of $i$, respectively. Let
\[
\mathcal{X}=\{\{i\in I\mid\theta\in\Theta_i\}\mid\theta\in\diag^+(A)\}.
\]
It is easy to see that $\mathcal{X}$ has the finite intersection property and thus generates a filter $\mathcal{F}$ over $I$.
For every $\theta\in\diag^+(A)$, since $\{i\in I\mid\theta\in\Theta_i\}\subseteq\{i\in I\mid B_i\models\theta\}$,
it follows that $\{i\in I\mid B_i\models\theta\}\in\mathcal{F}$,
and therefore $\prod_{\mathcal{F}}B_i\models\theta$ by Lemma~\ref{sh02}\ref{sh08}.
Hence, $\prod_{\mathcal{F}}B_i\models\diag^+(A)$.
For all atomic sentences $\xi$ such that $\neg\xi\in\diag^-(A)$
and all $\Theta\in[\diag^+(A)]^{<\omega}$, there is an $i\in I$ such that $\Theta_i=\Theta$ and $\Xi_i=\{\neg\xi\}$,
and thus $\bigcap_{\theta\in\Theta}\{i\in I\mid\theta\in\Theta_i\}\nsubseteq\{i\in I\mid B_i\models\xi\}$,
which implies that $\{i\in I\mid B_i\models\xi\}\notin\mathcal{F}$ and thus
$\prod_{\mathcal{F}}B_i\not\models\xi$ by Lemma~\ref{sh02}\ref{sh08},
that is, $\prod_{\mathcal{F}}B_i\models\neg\xi$. So $\prod_{\mathcal{F}}B_i\models\diag(A)$,
which implies that $A$ can be embedded into $\prod_{\mathcal{F}}B_i|_\Sigma$ by Lemma~\ref{sh03}\ref{sh04}.
Finally, for every $i\in I$, since $B_i\in\mathcal{K}'$, we have $B_i|_\Sigma\in\mathcal{K}$.
\end{proof}

\begin{theorem}\label{sh10}
A class $\mathcal{K}$ of $\Sigma$-structures is axiomatized by a set of basic Horn formulas
if and only if $\mathcal{K}$ is closed under isomorphisms, substructures, and reduced products.
\end{theorem}
\begin{proof}
If $\mathcal{K}$ is axiomatized by a set of basic Horn formulas,
then clearly $\mathcal{K}$ is closed under isomorphisms and substructures
since quantifier-free formulas are preserved by isomorphisms and substructures,
and it is closed under reduced products by Lemma~\ref{sh02}\ref{sh09}.
For the other direction, suppose that $\mathcal{K}$ is closed under isomorphisms, substructures, and reduced products.
Let $\Delta$ denote the set of all basic Horn formulas $\phi$ such that $\mathcal{K}\models\phi$. By Lemma~\ref{sh07},
for every $\Sigma$-structure $A$, if $A\models\Delta$, then $A$ can be embedded into some reduced product of structures in $\mathcal{K}$,
and therefore $A\in\mathcal{K}$. This means that $\mathcal{K}$ is axiomatized by $\Delta$.
\end{proof}

\begin{corollary}
A class $\mathcal{K}$ of $\Sigma$-structures is a quasivariety
if and only if $\mathcal{K}$ contains a unit and is closed under isomorphisms, substructures, and reduced products.
\end{corollary}
\begin{proof}
The necessity is obvious. For the sufficiency, suppose that $\mathcal{K}$ contains a unit
and is closed under isomorphisms, substructures, and reduced products. By Theorem~\ref{sh10},
$\mathcal{K}$ is axiomatized by a set $\Delta$ of basic Horn formulas, and since $\mathcal{K}$ contains a unit,
it follows that all formulas in $\Delta$ are strict.
\end{proof}

\subsection{A new choice-free proof of Birkhoff's HSP theorem}
\begin{lemma}\label{sh11}
Let $\mathcal{K}$ a class of $\Sigma$-structures and let $\Delta$ denote
the set of all atomic formulas $\phi$ such that $\mathcal{K}\models\phi$.
For every $\Sigma$-structure $A$, if $A\models\Delta$, then $A$ is a homomorphic
image of a substructure of a product of structures in $\mathcal{K}$.
\end{lemma}
\begin{proof}
Let $A$ be a $\Sigma$-structure such that $A\models\Delta$.
Let $\mathcal{K}'$ be the class of all $\Sigma(A)$-structures $B$ such that $B|_\Sigma\in\mathcal{K}$.
As in the proof of Lemma~\ref{sh07}, it is easily seen that every sentence in $\diag^-(A)$ has a model in $\mathcal{K}'$.
By Lemma~\ref{sh06}, there is a set
\[
I\subseteq\{\langle\neg\xi,B\rangle\mid\neg\xi\in\diag^-(A),B\in\mathcal{K}'\text{ and }B\models\neg\xi\}
\]
such that for every $\neg\xi\in\diag^-(A)$ there is $B$ for which $\langle\neg\xi,B\rangle\in I$.

For all $i\in I$, let $\neg\xi_i$ and $B_i$ be the first and second coordinates of $i$, respectively.
For every $\neg\xi\in\diag^-(A)$, there is an $i\in I$ such that $\xi_i=\xi$, and therefore $B_i\not\models\xi$,
which implies that $\prod B_i\not\models\xi$, that is, $\prod B_i\models\neg\xi$.
Hence, $\prod B_i\models\diag^-(A)$. By Lemma~\ref{sh03}\ref{sh05},
$A$ is a homomorphic image of a substructure of $\prod B_i|_\Sigma$.
Finally, for every $i\in I$, since $B_i\in\mathcal{K}'$, we have $B_i|_\Sigma\in\mathcal{K}$.
\end{proof}

\begin{theorem}\label{sh12}
A class $\mathcal{K}$ of $\Sigma$-structures is a variety
if and only if $\mathcal{K}$ is closed under homomorphic images, substructures, and products.
\end{theorem}
\begin{proof}
The necessity is obvious. For the sufficiency, suppose that $\mathcal{K}$ is closed under homomorphic images, substructures, and products.
Let $\Delta$ denote the set of all atomic formulas $\phi$ such that $\mathcal{K}\models\phi$. By Lemma~\ref{sh11},
for every $\Sigma$-structure $A$, if $A\models\Delta$, then $A$ is a homomorphic
image of a substructure of a product of structures in $\mathcal{K}$,
and so $A\in\mathcal{K}$. This means that $\mathcal{K}$ is axiomatized by $\Delta$, and thus, it is a variety.
\end{proof}

Note that, when $\Sigma$ contains no relation symbols, Theorem~\ref{sh12} is the usual Birkhoff's HSP theorem.
Note also that the proof given here differs from the proof in~\cite{Andreka1981} in that it does not rely on the concept of free algebras.

\end{document}